\documentclass[12pt,oneside]{article}

\usepackage{quotes}
\usepackage{amssymb}
\usepackage{amsmath}
\usepackage{tikz}
\usetikzlibrary{shapes,arrows,positioning}
\usepackage[letterpaper,left=1in,right=1in, bottom=1in, top=1in]{geometry}
\usepackage{titlesec}
\usepackage{setspace}
\usepackage{textcomp}
\usepackage{amsthm}
\usepackage{eso-pic}
\usepackage{graphicx}
\usepackage{color}
\usepackage{longtable}
\usepackage{transparent}
\usepackage{caption}
\usepackage{hyperref}

\renewcommand{\geq}{\geqslant}
\renewcommand{\leq}{\leqslant}

\newcommand{\tab}{\hspace{1cm}}

\newcommand{\ind}{\text{\normalfont{ind}}}

\titleformat{\section}{\bfseries\centering}{\thesection \hspace{0.5cm}}{12pt}{}

\titleformat{\subsection}{\bfseries}{\thesubsection \hspace{0.5cm}}{12pt}{}

\newtheorem{theorem}{Theorem}[section]

\newtheorem{corollary}{Corollary}[section]
\newtheorem{remark}{Remark}[section]
\newtheorem{definition}{Definition}[section]

\newtheorem{property}{Property}[section]

\begin{document}\setlength{\parindent}{0cm}\thispagestyle{empty}
	\begin{center}
		\Large \textbf{On Arithmetic Cordial Labeling of Some Graphs}
	\end{center}
	\vspace{2mm}
	Jason D. Andoyo$^1$, Jemina Clarisse C. Prudencio$^1$, and Ricky F. Rulete$^{1}$
	
	$^1$ Department of Mathematics and Statistics, University of Southeastern Philippines, Davao City, Philippines
	
	\section*{Abstract}\small
	\tab Let $\eta$ be a fixed positive integer. Let $S$ be a subset of $\mathbb{Z}$, $\star:S\times S\to \mathbb{Z}$ be a binary function, and $\zeta_{\eta}:\{\xi\in \mathbb{Z}:\gcd(\xi,\eta)=1\}\to \{0,1\}$ be a function. For a simple connected graph $G$ of order $n$, a bijective function $f:V(G)\to S$ (where $|S|=n$) is called an arithmetic cordial labeling modulo $\eta$ under $\langle S,\zeta_\eta,\star\rangle$ if the induced function $f_\eta^*:E(G)\to \{0,1\}$, defined by $f_\eta^*(uv)=0$ whenever $\zeta_\eta(f(a)\star f(b))=0$ or $\gcd(f(a)\star f(b),\eta)\neq 1$, and $f_\eta^*(uv)=1$ whenever $\zeta_\eta(f(a)\star f(b))=1$, satisfies the condition $|e_{f_\eta^*}(0)-e_{f_\eta^*}(1)|\leq 1$, where $e_{f_\eta^*}(i)$ is the number of edges with label $i$ ($i=0,1$). In this paper, we explore the arithmetic cordial labeling of some graphs under conditions imposed on the function $\zeta_\eta$. The graphs included are star graphs, ladder graphs, alternate cycle snake graphs, join graphs, corona graphs, and tensor product graphs.
	\normalsize
	\vspace{0.3cm}
	
	\textbf{Keywords:} binary function; graph; arithmetic cordial labeling
	
	\textbf{MSC 2020:} 05C76, 05C78, 11A05, 11A07
	\section{Introduction}
	
		\tab A simple graph $G=(V, E)$ is an ordered pair, where $V=V(G)$ is called the \textit{vertex set} and $E=E(G)$ is called the \textit{edge set}. The elements of $E$ are unordered pairs of distinct elements of $V$. The elements of $V$ and $E$ are called \textit{vertices} and \textit{edges}, respectively. If $G$ has $n$ vertices and $m$ edges, then $G$ has \textit{order} $n$ and \textit{size} $m$. Let $v\in V(G)$. Then the \textit{degree} of $v$, denoted by $\deg(v)$, is the number of vertices adjacent to $v$. If $\deg(v)=1$, then $v$ is called a \textit{pendant vertex} of $G$.
		
		\tab Graph labeling is a well-known concept in graph theory that studies the properties of assigning integers to the vertices and/or edges of a graph under specific conditions \cite{Gallian}. In 1987, I. Cahit \cite{Cahit} introduced the concept of \textit{cordial labeling}. This graph labeling assigns integers $0$ and $1$ to the vertices of a graph, where the label of each edge is the absolute difference of the labels of its endpoints. If the numbers of vertices labeled $0$ and $1$ differ by at most $1$, and the numbers of edges labeled $0$ and $1$ also differ by at most $1$, then the graph is said to admit a cordial labeling. This concept inspired the introduction of several variants of cordial labeling, including Legendre cordial labeling \cite{Andoyo1}, Euler cordial labeling \cite{Andoyo2}, Legendre product cordial labeling \cite{Prudencio}, logarithmic cordial labeling \cite{Andoyo3}, and $(a,b)$-Fibonacci-Legendre cordial labeling \cite{Andoyo4}. These variants use different concepts from number theory, such as the Legendre symbol, Euler's Theorem, discrete logarithms (indices), and the $(a,b)$-Fibonacci sequence. In this paper, we define a much more general concept that encompasses the above-mentioned variants of cordial labeling, which we call \textit{arithmetic cordial labeling}. 
		\newpage
		
		\tab Let $\eta$ be a fixed positive integer. In addition, let $S$ be a subset of $\mathbb{Z}$,
		$\star:S\times S\to \mathbb{Z}$ be a binary function, and $\zeta_\eta:\{\xi\in \mathbb{Z}:\gcd(\xi,\eta)=1\}\to \{0,1\}$ be a function. We call the triple $\langle S,\zeta_\eta,\star\rangle$ an \textit{arithmetic structure}. For a simple connected graph $G$ of order $n$, a bijective function $f : V(G) \to S$ (where $|S|=n$) is called an \textit{arithmetic cordial labeling modulo $\eta$ under $\langle S,\zeta_\eta,\star\rangle$} if the induced function $f_\eta^* : E(G) \to \{0,1\}$, defined by
		\begin{equation*}
			f_\eta^*(ab)=
			\begin{cases}
				0 & \text{if } \zeta_\eta(f(a)\star f(b))=0 \text{ or } \gcd(f(a)\star f(b),\eta)\neq 1\\
				1 & \text{if } \zeta_\eta(f(a)\star f(b))=1,
			\end{cases}
		\end{equation*}
		satisfies the condition $|e_{f_\eta^*}(0)-e_{f_\eta^*}(1)|\leq 1$, where $e_{f_\eta^*}(i)$ denotes the number of edges labeled $i$ $(i=0,1)$. A graph that admits this labeling is called an \textit{arithmetic cordial graph modulo $\eta$ under $\langle S,\zeta_\eta,\star\rangle$}.
		
		\tab Let $G$ be a simple connected graph of order $n$. Suppose that $(a/p)$ is the Legendre symbol of $a$ over an odd prime $p$ (see Definition 2.8 in \cite{Andoyo1}), $\ind_{\varpi,\eta}(a)$ is the discrete logarithm of $a$ to the base $\varpi$ modulo $\eta$ (where $\varpi$ is a fixed primitive root of $\eta$; see Definition 2.4 in \cite{Andoyo3}), and $F_i$ is the $i$th $(a,b)$-Fibonacci number (see \cite{Andoyo4}). Also, consider the expression $\frac{a^{\phi(\eta)}-1}{\eta}$ (from Euler's Theorem; see Theorem 2.12 in \cite{Andoyo2}), where $\eta$ is an odd positive integer with $\eta\geq 3$ and $\phi(\eta)$ is the Euler phi-function. Thus, the following are arithmetic cordial labeling with their corresponding arithmetic structures $\langle S,\zeta_\eta,\star\rangle$.
		
		\begin{enumerate}
			\item[$\bullet$] Legendre cordial labeling modulo $p$:
			$$S=\{1,2,\ldots,n\};\quad\zeta_p(a)=\frac{1+(a/p)}{2};\quad x\star y=x+y$$
			\item[$\bullet$] Euler cordial labeling modulo $\eta$:
			$$S=\{1,2,\ldots,n\},\quad\zeta_\eta(a)=\frac{a^{\phi(\eta)}-1}{\eta},\quad x\star y=x+y$$
			\item[$\bullet$] Legendre product cordial labeling modulo $p$:
			$$S=\{1,2,\ldots,n\};\quad\zeta_p(a)=\frac{1+(a/p)}{2};\quad x\star y=xy$$
			\item[$\bullet$] Logarithmic cordial labeling modulo $\eta$:
			$$S=\{1,2,\ldots,n\};\quad\ind_{\varpi,\eta}(a)\equiv \zeta_\eta(a)\pmod{2};\quad x\star y=x+y$$
			\item[$\bullet$] $(a,b)$-Fibonacci--Legendre cordial labeling modulo $p$:
			$$S=\{0,1,\ldots,n-1\};\quad\zeta_p(a)=\frac{1+(a/p)}{2};\quad x\star y=F_x+F_y$$
		\end{enumerate}
		
		\tab Now, suppose that the following properties hold for $\zeta_p$ where $p$ is an odd prime:
		
		\begin{property}\normalfont\label{prop1}
			Let $\gcd(\theta_1,p)=\gcd(\theta_2,p)=1$. If $\theta_1\equiv \theta_2\pmod{p}$, then $\zeta_p(\theta_1)=\zeta_p(\theta_2)$.
		\end{property}
		
		\begin{property}\normalfont\label{prop2}
			If $A_i=\{a:\zeta_p(a)=i,\ 1\leq a\leq p-1\}$ for $i=0,1$, then $|A_0|=|A_1|$. Hence, $|A_0|=|A_1|=\frac{p-1}{2}$.
		\end{property}
		
		\begin{property}\normalfont\label{prop3}
			Let $\gcd(\theta,p)=1$. Assume that $\chi_p:\{\xi\in\mathbb{Z}:\gcd(\xi,p)=1\}\to\{-1,1\}$ is a function (e.g., the Legendre symbol) and suppose that
			$$\zeta_p(\theta)=\frac{1+\chi_p(\theta)}{2}.$$
			Thus, $\chi_p(\theta_1\theta_2)=\chi_p(\theta_1)\chi_p(\theta_2)$.
		\end{property}
		
		\tab For Legendre cordial labeling and Legendre product cordial labeling, observe that if 
		\begin{equation}
			\zeta_p^1(a)=\frac{1+\chi_p(a)}{2}\label{eq0}
		\end{equation}
		where $\chi_p(a)=(a/p)$, then Properties~\ref{prop1}, \ref{prop2}, and \ref{prop3} hold for $\zeta_p^1$ (see Theorem 2.9 in \cite{Andoyo2} and Theorem 2.10 in \cite{Prudencio}). Additionally, for logarithmic cordial labeling, notice that if 
		\begin{equation}
			\ind_{\varpi_1,p}(a)\equiv \zeta_p^2(a)\pmod{2}\label{eq1}
		\end{equation}
		where $\varpi_1$ is a fixed primitive root of $p$, then $\zeta_p^2$ satisfies Properties~\ref{prop1} and \ref{prop2} (see Remarks 2 and 3 in \cite{Andoyo3} with $\eta=p$). Furthermore, let 
		\begin{equation}
			\zeta_p^3(a)=\frac{1+\chi_p(a)}{2}\label{eq2}
		\end{equation}
		with $\chi_p(a)=(-1)^{\ind_{\varpi_2,p}(a)}$, where $\varpi_2$ is a fixed primitive root of $p$. If $a_1\equiv a_2\pmod{p}$ with $\gcd(a_1,p)=\gcd(a_2,p)=1$, then by Remark 2 in \cite{Andoyo3}, we have $\ind_{\varpi_2,p}(a_1)=\ind_{\varpi_2,p}(a_2)$, which means $(-1)^{\ind_{\varpi_2,p}(a_1)}=(-1)^{\ind_{\varpi_2,p}(a_2)}$. Obviously, $\zeta_p^3(a_1)=\zeta_p^3(a_2)$, and so Property~\ref{prop1} holds. Now, assume that 
		$$A_i^2=\{a:\zeta_p^2(a)=i,\ 1\leq a\leq p-1\},\text{ and }$$
		$$A_i^3=\{a:\zeta_p^3(a)=i,\ 1\leq a\leq p-1\},$$
		for $i=0,1$ with $\varpi_1=\varpi_2$, where $\zeta_p^2$ is defined in (\ref{eq1}). If $a\in A_0^2$, then $\ind_{\varpi_2,p}(a)\equiv 0\pmod{2}$, and so $\chi_p(a)=1$. Hence, $a\in A_0^3$. Conversely, if $a\in A_0^3$, then 
		\begin{align*}
			1&=\frac{1+\chi_p(a)}{2}\tag{\text{by (\ref{eq2})}}\\
			2&=1+\chi_p(a)\\
			\chi_p(a)&=1\\
			(-1)^{\ind_{\varpi_2,p}(a)}&=1.
		\end{align*}
		Clearly, $\ind_{\varpi_2,p}(a)$ is even. Therefore, $a\in A_0^2$. Thus, $A_i^2=A_i^3$ for $i=0,1$, and it follows that $\zeta_p^3$ satisfies Property~\ref{prop2} since Property~\ref{prop2} holds for $\zeta_p^2$. Lastly, assume that $\gcd(a_1,p)=\gcd(a_2,p)=1$. Using Theorem 9.16 (ii) in \cite{Rosen} (page~369) with $m=p$ and $r=\varpi_2$, 
		$$\ind_{\varpi_2,p}(a_1a_2)\equiv \ind_{\varpi_2,p}(a_1)+\ind_{\varpi_2,p}(a_2)\pmod{(p-1)}.$$
		Since $p-1$ is even, we have 
		$$\ind_{\varpi_2,p}(a_1a_2)\equiv \ind_{\varpi_2,p}(a_1)+\ind_{\varpi_2,p}(a_2)\pmod{2},$$
		which means that $\ind_{\varpi_2,p}(a_1a_2)$ and $\ind_{\varpi_2,p}(a_1)+\ind_{\varpi_2,p}(a_2)$ have the same parity. Consequently,
		\begin{align*}
			\chi_p(a_1a_2)&=(-1)^{\ind_{\varpi_2,p}(a_1a_2)}\\
			&=(-1)^{\ind_{\varpi_2,p}(a_1)+\ind_{\varpi_2,p}(a_2)}\\
			&=(-1)^{\ind_{\varpi_2,p}(a_1)}(-1)^{\ind_{\varpi_2,p}(a_2)}\\
			&=\chi_p(a_1)\chi_p(a_2).
		\end{align*}
		Hence, Property~\ref{prop3} holds. Additionally, the function 
		$$\zeta_p^4(a)=\frac{1-\chi(a)}{2},$$
		where $\chi_p(a)=(a/p)$, also satisfies Properties~\ref{prop1}, \ref{prop2}, and \ref{prop3} with the same explanation as $\zeta_p^1$ defined in (\ref{eq0}).
		
		\tab As we observed, several functions satisfy Properties~\ref{prop1}, \ref{prop2}, and/or \ref{prop3}. Motivated by this, this paper explores the arithmetic cordial labeling of some graphs. In Sections~\ref{sec3} and \ref{sec4}, we assume that $\eta=p$ is an odd prime.
		
		\section{Basic Concepts}
		
		\begin{definition}\normalfont
			A \textit{path graph} $P_n$ of order $n$ is obtained from vertices $v_1,v_2,\ldots,v_n$ such that $v_i$ is adjacent to $v_{i+1}$ for $i=1,2,\ldots,n-1$. Also, a \textit{cycle graph} $C_n$ of order $n$ is created from a path graph $P_n$ such that two pendant vertices of $P_n$ are adjacent.
		\end{definition}
		
		\begin{definition}\normalfont
			A \textit{star graph} $\text{Star}_n$ of order $n$ is obtained from vertices $x_0,v_1,v_2,\ldots,v_{n-1}$ such that $x_0$ is adjacent to $v_i$ for $i=1,2,\ldots,n-1$. The vertex $x_0$ is called the \textit{central vertex} of $\text{Star}_n$.
		\end{definition}
		
		\begin{definition}\normalfont
			A \textit{complete graph} $K_n$ of order $n$ is a graph where two distinct vertices are adjacent. In addition, an \textit{empty graph} $\overline{K}_n$ of order $n$ is a graph with no edges.
		\end{definition}
		
		\begin{definition}\normalfont
			A \textit{ladder graph} $L_n$ of order $2n$ is obtained from two path graphs $P_n^1$ and $P_n^2$ such that the $i$th vertex of $P_n^1$ is adjacent to the $i$th vertex of $P_n^2$.
		\end{definition}
		
		\begin{definition}\normalfont
			A \textit{kayak paddle graph} $KP_{n,m,k}$ of order $n+m+k$ is created from two cycle graphs $C_n$ and $C_k$, and path graph $P_{m+2}$, for which a pendant vertex of $P_{m+2}$ is a vertex of $C_n$ and the other pendant vertex is a vertex of $C_k$. 
		\end{definition}
		
		\begin{definition}\normalfont
			An \textit{alternate cycle snake graph} $A_n(C_m)$ of order $nm$ is obtained from cycle graphs $C_m^1,C_m^2,\ldots,C_m^n$, with $V(C_m^i)=\{v_1^i,v_2^i,\ldots,v_n^i\}$ for $i=1,2,\ldots,n$, for which $v_m^j$ is adjacent to $v_1^{j+1}$ for $j=1,2,\ldots,n-1$.
		\end{definition}
		
		\begin{definition}\normalfont
			A graph $G$ is called a \textit{bipartite graph} if its vertex set $V(G)$ can be partitioned into two nonempty subsets $V_1$ and $V_2$ such that the edges of $G$ have one end in $V_1$ and one end in $V_2$. The sets $V_1$ and $V_2$ are called \textit{partite sets} of $G$.
		\end{definition}
		
		\begin{definition}\normalfont
			Let $G$ and $H$ be graphs. 
			\begin{enumerate}
				\item[i.] A \textit{join graph} $G+H$ is a graph with vertex set $V(G+H)=V(G)\cup V(H)$ and edge set $E(G+H)=E(G)\cup E(H)\cup \{ab:a\in V(G)\text{ and }b\in V(H)\}$.
				\item[ii.] A \textit{corona graph} $G\circ H$ is obtained from $|V(G)|$ copies of $H$ and one copy of $G$ such that every vertex of the $i$th copy of $H$ is adjacent to the $i$th vertex of $G$.
				\item[iii.] A \textit{tensor product graph} $G\times H$ is a graph with vertex set $V(G\times H)=V(G)\times V(H)$ and edge set $E(G\times H)=\{(a_1,b_1)(a_2,b_2):a_1a_2\in E(G)\text{ and }b_1b_2\in E(H)\}$.
			\end{enumerate}
		\end{definition}

		\begin{definition}\normalfont\label{equivalent}
			An arithmetic structure $\langle S_1,\zeta_{\eta_1}^1,\star_1\rangle$ is  \textit{equivalent} to another arithmetic structure $\langle S_2,\zeta_{\eta_2}^2,\star_2\rangle$, denoted by $\langle S_1,\zeta_{\eta_1}^1,\star_1\rangle\cong \langle S_2,\zeta_{\eta_2}^2,\star_2\rangle$, if there exists a bijective function $\psi:S_1\to S_2$ such that for all $a,b\in S_1$, $a\neq b$, 
			\begin{enumerate}
				\item[i.] $\gcd(a\star_1b,\eta_1)=1$ if and only if $\gcd(\psi(a)\star_2\psi(b),\eta_2)=1$, and
				\item[ii.] $\zeta_{\eta_1}^1(a\star_1 b)=\zeta_{\eta_2}^2(\psi(a)\star_2 \psi(b))$  whenever $\gcd(a\star_1 b,\eta_1)=1$.
			\end{enumerate}
		\end{definition}
		
		\section{General Results}
		\begin{theorem}\normalfont\label{thm3.1}
			Let $G$ be a simple connected graph and let $\eta_1$ and $\eta_2$ be fixed integers. Assume that $\langle S_1,\zeta_{\eta_1}^1,\star_1\rangle\cong \langle S_2,\zeta_{\eta_2}^2,\star_2\rangle$. Then $G$ is an arithmetic cordial graph modulo $\eta_1$ under $\langle S_1,\zeta_{\eta_1}^1,\star_1\rangle$ if and only if it is an arithmetic cordial graph modulo $\eta_2$ under $\langle S_2,\zeta_{\eta_2}^2,\star_2\rangle$.	
		\end{theorem}
		
		\begin{proof}
			Suppose that $G$ is an arithmetic cordial graph modulo $\eta_1$ under $\langle S_1,\zeta_{\eta_1}^1,\star_1\rangle$. So, there is a bijective function $f:V(G)\to S_1$ such that the induced edge label
			$$f_{\eta_1}^*(ab)=\begin{cases}
				0&\text{ if $\zeta_{\eta_1}^1(f(a)\star_1f(b))=0$ or $\gcd(f(a)\star_1 f(b),\eta_1)\neq 1$}\\
				1&\text{ if $\zeta_{\eta_1}^1(f(a)\star_1f(b))=1$}
			\end{cases}$$
			satisfies the condition $|e_{f_{\eta_1}^*}(0)-e_{f_{\eta_1}^*}(1)|\leq 1$. Since $\langle S_1,\zeta_{\eta_1}^1,\star_1\rangle\cong \langle S_2,\zeta_{\eta_2}^2,\star_2\rangle$, there exists a bijective function $\psi: S_1\to S_2$ such that for all $x,y\in S_1$, $\zeta_{\eta_1}^1(x\star_1y)=\zeta_{\eta_2}^2(\psi(x)\star_2\psi(y))$. Define a function $g:V(G)\to S_2$ by
			$$g(v)=\psi(f(v))$$
			for all $v\in V(G)$. Obviously, $g$ is a bijective function. Observe that the induced edge label for $g$ is
			$$g_{\eta_2}^*(ab)=\begin{cases}
				0&\text{ if $\zeta_{\eta_2}^2(g(a)\star_2g(b))=0$ or $\gcd(g(a)\star_2 g(b),\eta_2)\neq 1$}\\
				1&\text{ if $\zeta_{\eta_2}^2(g(a)\star_2g(b))=1$}
			\end{cases}.$$
			Because $g(a)\star_2 g(b)=\psi(f(a))\star_2 \psi(f(b))$, and $\zeta_{\eta_1}^1(x\star_1y)=\zeta_{\eta_2}^2(\psi(x)\star_2\psi(y))$, we have
			$$\zeta_{\eta_2}^2(g(a)\star_2g(b))=\zeta_{\eta_2}^2(\psi(f(a))\star_2 \psi(f(b)))=\zeta_{\eta_1}^1(f(a)\star_1f(b)).$$
			In addition, notice that $\gcd(x\star_1y,\eta_1)=1$ if and only if $\gcd(\psi(x)\star_2\psi(y),\eta_2)=1$ because $\langle S_1,\zeta_{\eta_1},\star_1\rangle\cong \langle S_2,\zeta_{\eta_2},\star_2\rangle$. So, $\gcd(f(a)\star_1f(b),\eta_1)\neq 1$ if and only if $\gcd(g(a)\star_2g(b),\eta_2)\neq 1$ since $g(v)=\psi(f(v))$. Therefore, $g_{\eta_2}^*(ab)=f_{\eta_1}^*(ab)$ for all $ab\in E(G)$. Hence, $e_{f_{\eta_1}^*}(i)=e_{g_{\eta_2}^*}(i)$ for $i=0,1$, and it follows that $|e_{g_{\eta_2}^*}(0)-e_{g_{\eta_2}^*}(1)|\leq 1$. Hence, $G$ is an arithmetic cordial graph modulo $\eta_2$ under $\langle S_2,\zeta_{\eta_2}^2,\star_2\rangle$. 
			
			\tab For the converse, by using the inverse of $\psi$, the proof is analogous. 
		\end{proof}

		\tab Let $\eta\geq 3$ be a positive integer and let $\varphi(\eta)=\{\xi\in \mathbb{Z}:\gcd(\xi,\eta)=1,\text{ }1\leq \xi< \eta\}$. In addition, suppose that $\phi(\eta)$ is the Euler phi-function of $\eta$. Clearly, $|\varphi(\eta)|=\phi(\eta)$. Consider the following properties:
		\begin{property}\normalfont(General version of Property~\ref{prop1})\label{prop4}
			Assume that $\gcd(\theta_1,\eta)=\gcd(\theta_2,\eta)=1$. If $\theta_1\equiv \theta_2\pmod{\eta}$, then $\zeta_\eta(\theta_1)=\zeta_\eta(\theta_2)$.
		\end{property}
		\begin{property}\normalfont(General version of Property~\ref{prop2})\label{prop5}
			If $A_i=\{a:\zeta_\eta(a)=i,\text{ }a\in \varphi(\eta)\}$ for $i=0,1$, then $|A_0|=|A_1|$. So, $|A_0|=|A_1|=\frac{\phi(\eta)}{2}$.
		\end{property}
		
		\begin{property}\normalfont(General version of Property~\ref{prop3})\label{prop6}
			Suppose that $\gcd(\theta,\eta)=1$. Assume that $\chi_\eta:\{\xi\in \mathbb{Z}:\gcd(\xi,\eta)=1\}\to \{-1,1\}$ is a function and let 	
			$$\zeta_\eta(\theta)=\frac{1+\chi_\eta(\theta)}{2}.$$
			Thus, $\chi_\eta(\theta_1\theta_2)=\chi_\eta(\theta_1)\chi_\eta(\theta_2)$.
		\end{property}
		
		\begin{theorem}\normalfont\label{thm3.2}
			Let $\eta\geq 3$ be a positive integer and let  Properties~\ref{prop4}, \ref{prop5}, and \ref{prop6} hold for $\zeta_\eta^j$, for $j=1,2$. Let $S=\{1,2,\ldots,\eta m-1\}$ where $m\geq 1$ is an integer. Then $\langle S,\zeta_\eta^1,\cdot\rangle\cong \langle S,\zeta_\eta^2,\cdot\rangle$.
		\end{theorem}
		
		\begin{proof}
			Let $A_i^j=\{a:\zeta_\eta^j(a)=i,\text{ }a\in \varphi(\eta)\}$ for $i=0,1$ and $j=1,2$. Let $T=\{1,2,\ldots,\eta\}$ and, by Property~\ref{prop5}, suppose that 
			\begin{align*}
				A_0^1&=\{q_1,q_2,\ldots,q_{\phi(\eta)/2}\},\\ A_1^1&=\{r_1,r_2,\ldots,r_{\phi(\eta)/2}\},\\ 
				T-(A_0^1\cup A_1^1)&=\{w_1,w_2,\ldots,w_{\eta-1-\phi(\eta)/2},\eta\},\\
				A_0^2&=\{s_1,s_2,\ldots,s_{\phi(\eta)/2}\},\\
				A_1^2&=\{t_1,t_2,\ldots,t_{\phi(\eta)/2}\},\\
				T-(A_0^2\cup A_1^2)&=\{x_1,x_2,\ldots,x_{\eta-1-\phi(\eta)/2},\eta\}.
			\end{align*}
			Define a function $\psi:S\to S$ as follows: for each $k=1,2,\ldots,m$,
			\begin{align}
				\psi(q_i+(k-1)\eta)&=s_i+(k-1)\eta\text{ for }i=1,2,\ldots,\frac{\phi(\eta)}{2},\label{EQ1}\\
				\psi(r_i+(k-1)\eta)&=t_i+(k-1)\eta\text{ for }i=1,2,\ldots,\frac{\phi(\eta)}{2},\label{EQ2}\\
				\psi(w_i+(k-1)\eta)&=x_i+(k-1)\eta\text{ for }i=1,2,\ldots,\eta-1-\frac{\phi(\eta)}{2}\label{EQ3}\\
				\psi(\eta+(k-1)\eta)&=\eta+(k-1)\eta\text{ for }k\neq m. \label{EQ3.5}
			\end{align}
			Hence, $\psi$ is a bijective function. It should be noted that all elements of $A_i^j$ are relatively prime to $\eta$, for $i=0,1$ and $j=1,2$. In addition, note that all elements of $T-(A_0^j\cup A_1^j)$ are not relatively prime to $\eta$, for $j=1,2$. Thus, by equations (\ref{EQ1}) and (\ref{EQ2}), and Property~\ref{prop4}, for any $a,b\in A_0^1\cup A_1^1$, we have 
			$$\gcd([a+(k_1-1)\eta][b+(k_2-1)\eta],\eta)=\gcd(ab,\eta)=1$$
			if and only if 
			$$\gcd(\psi(a+(k_1-1)\eta)\psi(b+(k_2-1)\eta),\eta)=\gcd(\psi(a)\psi(b),\eta)=1$$ 
			because $\psi(a),\psi(b)\in A_0^2\cup A_1^2$, for $k_1,k_2=1,2,\ldots,m$. Similarly, by equations~(\ref{EQ3}) and (\ref{EQ3.5}), and Property~\ref{prop4}, for any $c\in T$ and $d\in T-(A_0^1\cup A_1^1)$, we have
			$$\gcd([c+(k_1-1)\eta][d+(k_2-1)\eta],\eta)=\gcd(cd,\eta)\neq 1$$
			if and only if 
			$$\gcd(\psi(c+(k_1-1)\eta)\psi(d+(k_2-1)\eta),\eta)=\gcd(\psi(c)\psi(d),\eta)\neq 1$$ 
			since $\psi(c)\in T$ and $\psi(d)\in T-(A_0^2\cup A_1^2)$, for $k_1,k_2=1,2,\ldots,m$, but $k_1=1,2,\ldots,m-1$ whenever $c=\eta$, and $k_2=1,2,\ldots,m-1$ whenerver $d=\eta$. Thus, Definition~\ref{equivalent} (i) holds. 
			
			\tab Assume that 
			\begin{equation}
				\zeta_\eta^j(\theta)=\frac{1+\chi_\eta^j(\theta)}{2}\text{ or } \chi_\eta^j(\theta)=2\zeta_\eta^j(\theta)-1\label{EQ4}
			\end{equation}
			where $\chi_\eta^j:\{\xi\in \mathbb{Z}:\gcd(\xi,\eta)=1\}\to \{-1,1\}$ is a function, for $j=1,2$. By Properties~\ref{prop4} and \ref{prop6}, and equations~(\ref{EQ1}), (\ref{EQ2}), and (\ref{EQ4}), for any $a,b\in A_0^1\cup A_1^1$, we have
			\begin{align*}
				\zeta_\eta^1([a+(k_1-1)\eta][b+(k_2-1)\eta])&=\zeta_\eta^1(ab)\\
				&=\frac{1+\chi_\eta^1(a)\chi_\eta^1(b)}{2}\\
				&=\begin{cases}
					0&\text{ if $a\in A_0^1$ and $b\in A_1^1$, or $a\in A_1^1$ and $b\in A_0^1$}\\
					1&\text{ if $a,b\in A_0^1$ or $a,b\in A_1^1$},\text{ and }\\
				\end{cases}\\
				\zeta_\eta^2(\psi(a+(k_1-1)\eta)\psi(b+(k_2-1)\eta))&=\zeta_\eta^2(\psi(a)\psi(b))\\
				&=\frac{1+\chi_\eta^2(\psi(a))\chi_\eta^2(\psi(b))}{2}\\
				&=\begin{cases}
					0&\text{ if $a\in A_0^1$ and $b\in A_1^1$, or $a\in A_1^1$ and $b\in A_0^1$}\\
					1&\text{ if $a,b\in A_0^1$ or $a,b\in A_1^1$},
				\end{cases}
			\end{align*}
			for $k_1,k_2=1,2,\ldots,m$. Clearly, Definition~\ref{equivalent} (ii) holds. Therefore,  $\langle S,\zeta_\eta^1,\cdot\rangle\cong \langle S,\zeta_\eta^2,\cdot\rangle$. 
		\end{proof}
		
		\tab The proof works identically if we define $S = \{1,2,\dots,\eta m\}$; in that case, one only needs to remove the condition $k \neq m$ in equation~(\ref{EQ3.5}). Combining this with Theorems~\ref{thm3.1} and \ref{thm3.2}, we obtain the following corollary.
		
		\begin{corollary}\normalfont\label{cor3.1}
			Let $\eta\geq 3$ be an integer and let $G$ be a graph order $\eta m+\varepsilon$, where $m$ is a positive integer and $\varepsilon\in\{-1,0\}$. Suppose that Properties~\ref{prop4}, \ref{prop5}, and \ref{prop6} hold for $\zeta_\eta^j$, for $j=1,2$. Then $G$ is an arithmetic cordial graph modulo $\eta$ under $\langle S, \zeta_\eta^1,\cdot\rangle$ if and only if it is an arithmetic cordial graph modulo $\eta$ under $\langle S, \zeta_\eta^2,\cdot\rangle$, where $S=\{1,2,\ldots,\eta m+\varepsilon\}$.
		\end{corollary}
		
		\tab Using equations~(\ref{eq0}) and (\ref{eq2}) (where $\zeta_p^1$ and $\zeta_p^3$ satisfy Properties~\ref{prop1}, \ref{prop2}, and \ref{prop3}), Corollary~\ref{cor3.1} shows that a graph of order $mp+\varepsilon$, where $p$ is an odd prime, $m$ is a positive integer, and $\varepsilon \in \{-1,0\}$, is a Legendre product cordial graph modulo $p$ if and only if it is an arithmetic cordial graph modulo $p$ under $\langle S, \zeta_\eta^3, \cdot \rangle$, where $S = \{1,2,\ldots, mp+\varepsilon\}$.
		
		\begin{theorem}\normalfont (General version of Theorem 4.3 in \cite{Andoyo3})\label{thmgen4.3} Let $\eta\geq 3$ be an integer and let $G$ be a graph of order $n$ and size $m\phi(\eta)$, where $m$ is a positive integer. Also, suppose that Properties~\ref{prop4} and \ref{prop5} hold for $\zeta_\eta$. Additionally, let $f:V(G)\to \{1,2,\ldots,n\}$ be a bijective function and assume that $h(uv)=f(u)\star f(v)$ for every $uv\in E(G)$. If $f$ satisfies the following conditions:
			\begin{enumerate}
				\item[i.] $$\{h(uv):uv\in E(G)\}=\bigcup_{i\in T}\{j+i\eta:j\in \varphi(\eta)\}$$ 
				where $\varphi(\eta)=\{\xi:\gcd(\xi,\eta)=1,\text{ }0<\xi<\eta\}$ and $T\subseteq \mathbb{Z}$ with $|T|=m$,
				\item[ii.] $h(uv)\neq h(ab)$ for every $uv\neq ab$,
			\end{enumerate}
			then $G$ is an arithmetic cordial graph modulo $\eta$ under $\langle S,\zeta_\eta,\star\rangle$, where $S\subseteq \mathbb{Z}$ with $|S|=n$.
		\end{theorem}
		
		\begin{proof}
			Suppose that $\vartheta_i=\{j+i\eta:j\in\varphi(\eta)\}$ for $i\in T$. In addition, assume that $\vartheta_i\pmod{\eta}=\{\xi:a\equiv \xi\pmod{\eta},\text{ }a\in \vartheta_i,\text{ }0<\xi<\eta\}$ for $i\in T$. By Property~\ref{prop4}, we have
			$$\vartheta_i\pmod{\eta}=\varphi(\eta)$$
			for $i\in T$. By conditions (i) and (ii), together with Property~\ref{prop5}, we have
			$$e_{f_\eta^*}(0)=e_{f_\eta^*}(1)=\frac{m\phi(\eta)}{2}$$
			and so $|e_{f_\eta^*}(0)-e_{f_\eta^*}(1)|=0$. Therefore, $G$ is an arithmetic cordial graph modulo $\eta$ under $\langle S,\zeta_\eta,\star\rangle$, where $S\subseteq \mathbb{Z}$ with $|S|=n$.
		\end{proof}
		\tab Consider the Jacobi symbol $(a/n)$, which is defined as follows. Let $n\geq 1$ be an odd integer and let $a$ be an integer relatively prime to $n$. Then the Jacobi symbol is
		$$(a/n)=\prod_{i=1}^{t}(a/p_i)^{e_i}$$
		where $n=\prod_{i=1}^{t}p_i^{e_i}$ is the prime factorization of $n$, and the symbols on the right-hand side are Legendre symbols. This symbol generalizes the Legendre symbol by incorporating positive odd integers. Assume that
		\begin{equation}
			\zeta_\eta^5(a)=\frac{1+(a/\eta)}{2}\label{eq5}
		\end{equation}
		where $(a/\eta)$ is a Jacobi symbol. It is evident that $\zeta_\eta^5$ satisfies Properties~\ref{prop4} and~\ref{prop5} whenever $\eta$ is not a perfect square (see Theorem~11.10 (i) in~\cite{Rosen} (page 444) and Proposition~11 in~\cite{Mondal}). In fact, the function $\zeta_p^1$ defined in~(\ref{eq0}) is a special case of $\zeta_\eta^5$ when $\eta=p$, where $p$ is an odd prime. As an application of Theorem~\ref{thmgen4.3}, we obtain the following result.
		
		\begin{theorem}\normalfont
			Let $\eta\geq 3$ be an integer that is not a perfect square, and suppose that $c$ and $m$ are positive integers. The graph $\text{Star}_{m\phi(\eta)+1}$ is an arithmetic cordial graph modulo $\eta$ under $\langle S,\zeta_\eta^5,+\rangle$, where $S=\{c\eta\}\cup \left(\bigcup_{j=0}^{m-1}\{i+j\eta:i\in \varphi(\eta)\}\right)$ and the function $\zeta_\eta^5$ is defined in~(\ref{eq5}). 
		\end{theorem}
		
		\begin{proof}
			Note that the size of $\text{Star}_{m\phi(\eta)+1}$ is $m\phi(\eta)$. Let $V(\text{Star}_{m\phi(\eta)+1})=\{v_0\}\cup \{v_i:i\in W\}$, where $W=\bigcup_{j=0}^{m-1}\{i+j\eta:i\in\varphi(\eta)\}$ and $v_0$ is the central vertex. Define a function $f:V(\text{Star}_{m\phi(\eta)+1})\to S$ by
			\begin{align*}
				f(v_0)&=c\eta,\\
				f(v_i)=i&\text{ for }i\in W,
			\end{align*}
			where $S=\{c\eta\}\cup \left(\bigcup_{j=0}^{m-1}\{i+j\eta:i\in \varphi(\eta)\}\right)$. Thus, $f$ is a bijective function. 
			
			\tab For the edges, we have
			$$f(v_0)+f(v_i)=c\eta+i+j\eta=i+(c+j)\eta$$
			for $i\in \varphi(\eta)$ and $j=0,1,\ldots,m-1$. Let $T=\{c+j:j=0,1,\ldots,m-1\}$. Hence,
			$$\{f(v_0)+f(v_i):i\in W\}=\bigcup_{k\in T}\{i+k\eta:i\in\varphi(\eta)\}.$$ 
			It is clear that conditions (i) and (ii) of Theorem~\ref{thmgen4.3} are satisfied. Thus, $\text{Star}_{m\phi(\eta)+1}$ is an arithmetic cordial graph modulo $\eta$ under $\langle S,\zeta_\eta^5,+\rangle$.
		\end{proof}

		\section{Arithmetic Cordial Labeling modulo $p$ under $\left\langle S,\zeta_p,+\right\rangle$}\label{sec3}
		\tab In this section, we assume that Properties~\ref{prop1} and \ref{prop2} hold for $\zeta_p$. 
		
		\begin{theorem}\normalfont
			The ladder graph $L_p$ is an arithmetic cordial graph modulo $p$ under $\langle S,\zeta_p,+\rangle$ where $S=\{1,2,\ldots,2p\}$, for all $p\geq 3$.
		\end{theorem}
		
		\begin{proof}
			Assume that $P_p^1$ and $P_p^2$ are path graphs of $L_p$ with $V(P_p^j)=\{v_1^j,v_2^j,\ldots,v_p^j\}$ for $j=1,2$. Let $f:V(G)\to S$ be a function defined by 
			$$f(v_i^j)=\begin{cases}
				i+\frac{p-1}{2}+(j-1)p&\text{ for }i=1,2,\ldots,\frac{p+1}{2}\\
				i-\frac{p+1}{2}+(j-1)p&\text{ for }i=\frac{p+3}{2},\frac{p+5}{2},\ldots,p,
			\end{cases}$$
			for $j=1,2$, where $S=\{1,2,\ldots,2p\}$. Clearly, $f$ is a bijective function.
			
			\tab For the edges of $P_p^j$, 
			$$f(v_i^j)+f(v_{i+1}^j)\equiv \begin{cases}
				2i\pmod{p}&\text{ for }i=1,2,\ldots,\frac{p-1}{2}\\
				2i-p\pmod{p}&\text{ for }i=\frac{p+1}{2},\frac{p+3}{2},\ldots,p-1
			\end{cases}$$ 
			for $j=1,2$, and for the remaining edges,
			$$f(v_i^1)+f(v_i^2)\equiv \begin{cases}
				2i-1\pmod{p}&\text{ for }i=1,2,\ldots,\frac{p-1}{2}\\
				0\pmod{p}&\text{ for }i=\frac{p+1}{2}\\
				2i-p-1\pmod{p}&\text{ for }i=\frac{p+3}{2},\frac{p+5}{2},\ldots,p.
			\end{cases}$$
			In view of this labeling, let 
			$$\alpha_j=\bigcup_{i=1}^{p-1}\{\xi:f(v_i^j)+f(v_{i+1}^j)\equiv \xi\pmod{p},\text{ }0< \xi<p\}$$
			for $j=1,2$, and let
			$$\beta=\bigcup_{\substack{i=1\\i\neq (p+1)/2}}^p\{\xi:f(v_i^1)+f(v_i^2)\equiv \xi\pmod{p},\text{ }0<\xi<p\}.$$
			Evidently, $\alpha_1=\alpha_2=\beta=\{1,2,\ldots,p-1\}$. By Property~\ref{prop2}, together with the fact that $f_p^*(v_{(p+1)/2}^1v_{(p+1)/2}^2)=0$, we have
			$$e_{f_p^*}(0)=2\left(\frac{p-1}{2}\right)+1=p\text{ and }e_{f_p^*}(1)=2\left(\frac{p-1}{2}\right)=p-1.$$
			Therefore, $|e_{f_p^*}(0)-e_{f_p^*}(1)|=1$ and it follows that $L_p$ is an arithmetic cordial graph modulo $p$ under $\langle S,\zeta_p,+\rangle$.
		\end{proof}

		\begin{theorem}\normalfont
			The alternate cycle snake graph $A_n(C_p)$ is an arithmetic cordial graph modulo $p$ under $\langle S,\zeta_p,+\rangle$ where $\zeta_p(1)=1$ and $S=\{1,2,\ldots,np\}$, for all $n\geq 2$ and $p\geq 3$.
		\end{theorem}
		
		\begin{proof}
			Assume that $C_p^1,C_p^2,\ldots,C_p^n$ are cycle graphs of $A_n(C_p)$ with $V(C_p^j)=\{v_1^j,v_2^j,\ldots,v_p^j\}$ for $j=1,2,\ldots,n$. Now, define a function $f:V(A_n(C_p))\to S$ by
			$$f(v_i^j)=i+(j-1)p$$
			for $i=1,2,\ldots,p$ and $j=1,2,\ldots,n$, where $S=\{1,2,\ldots,np\}$. Hence, $f$ is a bijective function.
			
			\tab For the edges of $C_p^j$, we have
			$$f(v_i^j)+f(v_{i+1}^j)\equiv \begin{cases}
				2i+1\pmod{p}&\text{ for }i=1,2,\ldots,\frac{p-3}{2}\\
				0\pmod{p}&\text{ for }i=\frac{p-1}{2}\\
				2i-p+1\pmod{p}&\text{ for }i=\frac{p+1}{2},\frac{p+3}{2},\ldots,p-1,
			\end{cases}$$
			and $f(v_1^j)+f(v_p^j)\equiv 1\pmod{p}$, for $j=1,2,\ldots,n$. Additionally, for the remaining edges, 
			$$f(v_p^k)+f(v_1^{k+1})\equiv 1\pmod{p}$$
			for $k=1,2,\ldots,n-1$. 
			
			\tab In view of the above labeling, for each $j=1,2,\ldots,n$, assume that 
			$$\alpha_j=\bigcup_{\substack{i=1\\i\neq (p-1)/2}}^{p-1}\{\xi:f(v_i^j)+f(v_{i+1}^j)\equiv \xi\pmod{p},\text{ }0<\xi<p\},\text{ and }$$
			$$\beta_j=\{\xi:f(v_1^j)+f(v_p^j)\equiv 1\pmod{p},\text{ }0<\xi<p\}.$$
			Consequently, $\alpha_j\cup\beta_j=\{1,2,\ldots,p-1\}$ for $j=1,2,\ldots,n$. By Property~\ref{prop2}, there are $\frac{p-1}{2}$ edges of $C_p^j$ with label $0$ and $\frac{p-1}{2}$ edges with label $1$, under $f_p^*$, for $j=1,2,\ldots,n$. Moreover, observe that $f_p^*(v_{(p-1)/2}^jv_{(p-1)/2}^j)=0$ for $j=1,2,\ldots,n$, and $f_p^*(v_1^kv_p^{k+1})=1$ for $k=1,2,\ldots,n-1$ (since $\zeta_p(1)=1$). Therefore,  
			$$e_{f_p^*}(0)=n\left(\frac{p-1}{2}\right)+n\text{ and }e_{f_p^*}(1)=n\left(\frac{p-1}{2}\right)+n-1.$$
			Consequently, $|e_{f_p^*}(0)-e_{f_p^*}(1)|=1$. Hence, $A_n(C_p)$ is an arithmetic cordial graph modulo $p$ under $\langle S,\zeta_p,+\rangle$.
		\end{proof}
		
		\begin{theorem}\normalfont
			Let $G$ be a connected graph of order $n$ and assume that $\varepsilon\in \{-1,0,1\}$. If $G$ has size $n+\varepsilon$, then $G\circ P_{p-1}$ is an arithmetic cordial graph modulo $p$ under $\left\langle S,\zeta_p,+\right\rangle$, where $\zeta_p(1)=1$, $\zeta_p(2)=0$, and $S=\{1,2,\ldots,np\}$, for all $n\geq 2$ and $p\geq 3$.
		\end{theorem}
		
		\begin{proof}
			Let $V(G)=\{v_1,v_2,\ldots,v_n\}$ and assume that $P_{p-1}^i$ is the $i$th copy of $P_{p-1}$ with $V(P_{p-1}^i)=\{u_1^i,u_2^i,\ldots,u_{p-1}^i\}$ such that each vertex of $P_{p-1}^i$ is adjacent to $v_i$, for $i=1,2,\ldots,n$. Now, let $f:V(G\circ P_{p-1})\to S$ be a function defined by
			\begin{align*}
				f(u_j^i)&=\begin{cases}
					j+\frac{p+1}{2}+(i-1)p&\text{ for }j=1,2,\ldots,\frac{p-1}{2}\\
					j-\frac{p-1}{2}+(i-1)p&\text{ for }j=\frac{p+1}{2},\frac{p+3}{2},\ldots,p-1,
				\end{cases}\\
				f(v_i)&=\frac{p+1}{2}+(i-1)p,
			\end{align*}
			for $i=1,2,\ldots,n$, where $S=\{1,2,\ldots,np\}$.
			
			\tab For the edges of $P_{p-1}^i$, we have
			$$f(u_j^i)+f(u_{j+1}^i)\equiv \begin{cases}
				2j+2\pmod{p}&\text{ for }j=1,2,\ldots,\frac{p-3}{2}\\
				2j+2-p\pmod{p}&\text{ for }j=\frac{p-1}{2},\frac{p+1}{2},\ldots,p-2
			\end{cases}$$
			for $i=1,2,\ldots,n$. With this labeling, for each $i=1,2,\ldots,n$, let 
			$$\alpha_i=\bigcup_{j=1}^{p-2}\{\xi:f(u_j^i)+f(u_{j+1}^i)\equiv \xi\pmod{p},\text{ }0<\xi<p\}$$
			and it follows that $\alpha_i=\{1\}\cup \{3,4,\ldots,p-1\}$. Since $2\in A_0$ and $2\notin \alpha_i$, by Property~\ref{prop2}, $\frac{p-1}{2}-1$ edges of $P_{p-1}^i$ labeled $0$ and $\frac{p-1}{2}$ edges are labeled $1$, under $f_p^*$, for each $i=1,2,\ldots,n$. 
			
			\tab For the edges of the form $u_j^iv_i$, we have
			\begin{align*}
				f(u_j^i)+f(v_i)&\equiv j+1 \pmod{p}\text{ for }j=1,2,\ldots,p-1,\\
				f(u_{p-1}^i)+f(v_i)&\equiv 0\pmod{p}
			\end{align*}
			for $i=1,2,\ldots,n$. In view of this labeling, let 
			$$\beta_i=\bigcup_{j=1}^{p-2}\{\xi:f(u_j^i)+f(v_i)\equiv \xi\pmod{p},\text{ }0<\xi<p\}$$
			and it follows that $\beta_i=\{2,3,\ldots,p-1\}$, for $i=1,2,\ldots,n$. Note that $1\in A_1$ but $1\notin \beta_i$, for $i=1,2,\ldots,n$.  In addition, observe that $f_p^*(u_{p-1}^iv_i)=0$ for $i=1,2,\ldots,n$. Combining this with Property~\ref{prop2}, we have $n\left(\frac{p-1}{2}\right)+n$ edges of the form $u_j^iv_i$ with label $0$ and $n\left(\frac{p-1}{2}-1\right)$ edges with label $1$, under $f_p^*$.
			
			\tab Lastly, for the edges of $G$, observe that 
			$$f(a)+f(b)\equiv 2\pmod{p}$$
			for any $ab\in E(G)$. Since $\zeta_p(2)=1$, it follows that all edges of $G$ have label $1$, under $f_p^*$. 
			
			\tab Therefore, 
			$$e_{f_p^*}(0)=n\left(\frac{p-1}{2}-1\right)+n\left(\frac{p-1}{2}\right)+n\text{ and }$$
			$$e_{f_p^*}(1)=n\left(\frac{p-1}{2}\right)+n\left(\frac{p-1}{2}-1\right)+n+\varepsilon.$$
			Since $\varepsilon\in\{-1,0,1\}$, we have $|e_{f_p^*}(0)-e_{f_p^*}(1)|\leq 1$. So, $G\circ P_{p-1}$ is an arithmetic cordial graph modulo $p$ under $\langle S,\zeta_p,+\rangle$.
		\end{proof}
		
		\begin{theorem}\normalfont
			Let $G$ be a connected bipartite graph of order $n$. Then the tensor product graph $K_p\times G$ is an arithmetic cordial graph modulo $p$ under $\left\langle S,\zeta_p,+\right\rangle$ where $S=\{1,2,\ldots,np\}$, for all $n\geq 2$ and $p\geq 3$.
		\end{theorem}
		
		\begin{proof}
			Let $V(K_p)=\{v_1,v_2,\ldots,v_p\}$ and let $m$ be the size of $G$. Since $G$ is a bipartite graph, its vertex set $V(G)$ can be partitioned into two nonempty subsets $V_1$ and $V_2$ with $V_j=\{u_1^j,u_2^j,\ldots,u_{n_j}^j\}$ for $j=1,2$, where $n_1+n_2=n$. Thus,
			$$E(K_p\times G)=\bigcup_{u_x^1u_y^2\in E(G)}\{(v_i,u_x^1)(v_k,u_y^2):i,k=1,2,\ldots,p,\text{ }i\neq k\}.$$
			\tab Let $f:V(K_p\times G)\to S$ be a function, where $S=\{1,2,\ldots,np\}$. Define $f$ as follows: for each $t_j=1,2,\ldots,n_j$,
			$$f((v_i,u_{t_j}^j))=\begin{cases}
				i+(t_1-1)p&\text{ for }j=1\text{ and }i=1,2,\ldots,p\\
				p-i+(n_1+t_2-1)p&\text{ for }j=2\text{ and }i=1,2,\ldots,p-1\\
				p+(n_1+t_2-1)p&\text{ for }j=2\text{ and }i=p.
			\end{cases}$$
			Hence, $f$ is a bijective function. For the edges,
			for each $u_x^1u_y^2\in E(G)$ and each $i=1,2,\ldots,p$ with $i\neq k$, we have
			$$f((v_i,u_x^1))+f((v_k,u_y^2))\equiv \begin{cases}
				i-k\pmod{p}&\text{ for }k=1,2,\ldots,p-1\\
				i\pmod{p}&\text{ for }k=p.
			\end{cases}$$
			In accordance with the above labeling, for each $u_x^1u_y^2\in E(G)$ and each $i=1,2,\ldots,p$, let 
			$$\alpha_{u_x^1u_y^2}^j=\bigcup_{\substack{k=1\\k\neq i}}^{p}\{\xi:f((v_i,u_x^1))+f((v_k,u_y^2))\equiv\xi\pmod{p},\text{ }0<\xi<p\}$$
			and as a result $\alpha_{u_x^1u_y^2}^j=\{1,2,\ldots,p-1\}$. By Property~\ref{prop2}, we have
			$$e_{f_p^*}(0)=e_{f_p^*}(1)=mp\left(\frac{p-1}{2}\right)$$
			and it follows that $|e_{f_p^*}(0)-e_{f_p^*}(1)|=0$. Therefore, $K_p\times G$ is an arithmetic cordial graph modulo $p$ under $\langle S,\zeta_p,+\rangle$.
		\end{proof}
		\section{Arithmetic Cordial Labeling modulo $p$ under $\left\langle S,\zeta_p,\cdot\right\rangle$}\label{sec4}
		\tab In this section, we assume that Properties~\ref{prop1}, \ref{prop2}, and \ref{prop3} hold for $\zeta_p$. Also, let $\zeta_p(a)=\frac{1+\chi_p(a)}{2}$ where $\chi_p:\{\xi\in \mathbb{Z}:\gcd(\xi,p)=1\}\to\{-1,1\}$ is a function. By Property~\ref{prop2}, suppose that $s_1,s_2,\ldots,s_{(p-1)/2}\in A_0$ and $r_1,r_2,\ldots,r_{(p-1)/2}\in A_1$.
		
		\begin{remark}\normalfont\label{rem5.1}
			Let $a$ and $b$ be integers relatively prime to $p$. Then
			\begin{equation}
				\chi_p(ab)=\begin{cases}
					1&\text{ if }\chi_p(a)=\chi_p(b)=1\text{ or }\chi_p(a)=\chi_p(b)=-1\\
					-1&\text{ if }\chi_p(a)=-\chi_p(b).
				\end{cases}\label{eqrem}
			\end{equation}
			Thus,
			$$\zeta_p(ab)=\begin{cases}
				0&\text{ if }a,b\in A_0\text{ or }a,b\in A_1\\
				1&\text{ if }a\in A_0\text{ and }b\in A_1,\text{ or }a\in A_1\text{ and }b\in A_0. 
			\end{cases}$$
		\end{remark}
		
		\begin{proof}
			Immediately follows Properties~\ref{prop2} and \ref{prop3}.
		\end{proof}
		
		\begin{theorem}\normalfont
			The join graph $\overline{K}_p+KP_{(p-1)/2,0,(p-1)/2}$ is an arithmetic cordial graph modulo $p$ under $\left\langle S,\zeta_p,\cdot\right\rangle$ where $S=\{1,2,\ldots,2p-1\}$, for all $p\geq 7$.
		\end{theorem}
		
		\begin{proof}
			Let $C_{(p-1)/2}^1$ and $C_{(p-1)/2}^2$ be cycles, and $P_2$ be a path graph of $KP_{(p-1)/2,0,(p-1)/2}$. Suppose that 
			\begin{align*}
				V(\overline{K}_p)&=\{v_1,v_2,\ldots,v_p\},\\
				V(C_{(p-1)/2}^1)&=\{u_1,u_2,\ldots,u_{(p-1)/2}\},\\
				V(C_{(p-1)/2}^2)&=\{w_1,w_2,\ldots,w_{(p-1)/2}\},\\
				V(P_2)&=\{u_{(p-1)/2},w_1\}.
			\end{align*}
			Define the function $f:V(\overline{K}_p+KP_{(p-1)/2,0,(p-1)/2})\to S$ by
			\begin{align*}
				f(v_i)&=\begin{cases}
					r_i&\text{ for }i=1,2,\ldots,\frac{p-1}{2}\\
					s_{i-(p-1)/2}&\text{ for }i=\frac{p+1}{2},\frac{p+3}{2},\ldots,p-1\\
					p&\text{ for }i=p,
				\end{cases}\\
				f(u_i)&=r_i+p\text{ for }i=1,2,\ldots,\frac{p-1}{2}\\
				f(w_i)&=s_i+p\text{ for }i=1,2,\ldots,\frac{p-1}{2},
			\end{align*}
			where $S=\{1,2,\ldots,2p-1\}$. Then $f$ is a bijective function.
			
			\tab For the edges of $C_{(p-1)/2}^1$, $C_{(p-1)/2}^2$, and $P_2$,
			\begin{align*}
				f(u_i)f(u_{i+1})&\equiv r_ir_{i+1}\pmod{p}\text{ for }i=1,2,\ldots,\frac{p-3}{2},\\
				f(u_1)f(u_{(p-1)/2})&\equiv r_1r_{(p-1)/2}\pmod{p},\\
				f(w_i)f(w_{i+1})&\equiv s_is_{i+1}\pmod{p}\text{ for }i=1,2,\ldots,\frac{p-3}{2},\\
				f(w_1)f(w_{(p-1)/2})&\equiv s_1s_{(p-1)/2}\pmod{p},\\
				f(u_{(p-1)/2})f(w_1)&\equiv r_{(p-1)/2}s_1\pmod{p}.
			\end{align*}
			By Remark~\ref{rem5.1},
			\begin{align*}
				f_p^*(u_iu_{i+1})&=1\text{ for }i=1,2,\ldots,\frac{p-3}{2},\\
				f_p^*(u_1u_{(p-1)/2})&=1,\\
				f_p^*(w_iw_{i+1})&=1\text{ for }i=1,2,\ldots,\frac{p-3}{2},\\
				f_p^*(w_1w_{(p-1)/2})&=1,\\
				f_p^*(u_{(p-1)/2}w_1)&=0.
			\end{align*}
			
			\tab For the edges of the forms $v_iu_j$ and $v_iw_j$, for each $j=1,2,\ldots,\frac{p-1}{2}$,
			\begin{align*}
				f(v_i)f(u_j)&\equiv \begin{cases}
					r_ir_j\pmod{p}&\text{ for }i=1,2,\ldots,\frac{p-1}{2}\\
					s_{i-(p-1)/2}r_j\pmod{p}&\text{ for }i=\frac{p+1}{2},\frac{p+3}{2},\ldots,p-1\\
					0\pmod{p}&\text{ for }i=p,
				\end{cases}\\
				f(v_i)f(w_j)&\equiv \begin{cases}
					r_is_j\pmod{p}&\text{ for }i=1,2,\ldots,\frac{p-1}{2}\\
					s_{i-(p-1)/2}s_j\pmod{p}&\text{ for }i=\frac{p+1}{2},\frac{p+3}{2},\ldots,p-1\\
					0\pmod{p}&\text{ for }i=p.
				\end{cases}
			\end{align*}
			By Remark~\ref{rem5.1}, for each $j=1,2,\ldots,\frac{p-1}{2}$,
			\begin{align*}
				f_p^*(v_iu_j)&= \begin{cases}
					1&\text{ for }i=1,2,\ldots,\frac{p-1}{2}\\
					0&\text{ for }i=\frac{p+1}{2},\frac{p+3}{2},\ldots,p,
				\end{cases}\\
				f_p^*(v_iw_j)&= \begin{cases}
					0&\text{ for }i=1,2,\ldots,\frac{p-1}{2},p\\
					1&\text{ for }i=\frac{p+1}{2},\frac{p+3}{2},\ldots,p-1.
				\end{cases}
			\end{align*}
			
			\tab Therefore,
			\begin{align*}
				e_{f_p^*}(0)&=1+2\left(\frac{p+1}{2}\right)\left(\frac{p-1}{2}\right)=1+(p-1)\left(\frac{p+1}{2}\right),\text{ and }\\
				e_{f_p^*}(1)&=p-1+2\left(\frac{p-1}{2}\right)^2=(p-1)\left(\frac{p+1}{2}\right).
			\end{align*}
			Clearly, $|e_{f_p^*}(0)-e_{f_p^*}(1)|=1$ and so $\overline{K}_p+KP_{(p-1)/2,0,(p-1)/2}$ is an arithmetic cordial graph modulo $p$ under $\left\langle S,\zeta_p,\cdot\right\rangle$.
		\end{proof}

		\begin{theorem}\normalfont\label{thmjoin}
			Let $G_1$ and $G_2$ be graphs of the same order $\frac{p-1}{2}$ and let $f:V(G_1+G_2)\to S$ be a bijective function, where $S=\{1,2,\ldots,p-1\}$. For each $i=1,2$ and each $j=-1,1$, define
			$$\Omega_i^j=\{ab\in E(G_i):\chi_p(f(a)f(b))=j\}$$
			and suppose that
			$$B=\{a\in V(G_1):\chi_p(f(a))=1\}.$$
			In addition, let $\epsilon\in \{-1,0,1\}$. If 
			\begin{equation}
				|\Omega_1^1|+|\Omega_2^1|=|\Omega_1^{-1}|+|\Omega_2^{-1}|+\left(2|B|-\frac{p-1}{2}\right)^2+\epsilon,\label{eqjoin}
			\end{equation}
			then the join graph $G_1+G_2$ is an arithmetic cordial graph modulo $p$ under $\left\langle S,\zeta_p,\cdot\right\rangle$.
		\end{theorem}
		
		\begin{proof}
			Note that $S=A_0\cup A_1$. From the hypothesis, it is clear that $|\Omega_i^{-1}|$ edges of $G_i$ are labeled $0$ and
			$|\Omega_i^{1}|$ edges are labeled $1$ under $f_p^*$, for $i=1,2$. Let $k=|B|$.
			Clearly, $k$ vertices of $G_1$ have label $r_i$ for some $i$. It means that the remaining vertices have label $s_i$, and there are
			$\frac{p-1}{2}-k$ such vertices. Consequently, $k$ vertices of $G_2$ have label $s_i$ and
			$\frac{p-1}{2}-k$ vertices have label $r_i$ for some $i$. By Remark~\ref{rem5.1}, there are $2k\left(\frac{p-1}{2}-k\right)$ edges labeled $1$
			and $k^2+\left(\frac{p-1}{2}-k\right)^2$ edges labeled $0$ of the form $ab$, where $a\in V(G_1)$ and $b\in V(G_2)$, under $f_p^*$. 
			
			\tab Consequently,
			\begin{align*}
				e_{f_p^*}(1) &= 2k\left(\frac{p-1}{2}-k\right)+|\Omega_{1}^{1}|+|\Omega_{2}^{1}|
				\text{ and }\\
				e_{f_p^*}(0) &= k^2+\left(\frac{p-1}{2}-k\right)^2+|\Omega_{1}^{-1}|+|\Omega_{2}^{-1}|
			\end{align*}
			Given that $$|\Omega_1^1|+|\Omega_2^1| = |\Omega_1^{-1}| + |\Omega_2^{-1}|+\left(2k-\frac{p-1}{2}\right)^2 +\epsilon$$ 
			where $k = |B|$ and $\epsilon\in\{-1,0,1\}$, we have  
			\begin{align*}
				e_{f_p^*}(0)-e_{f_p^*}(1) & = 
				k^2+\left(\frac{p-1}{2}-k\right)^2+|\Omega_1^{-1}|+|\Omega_2^{-1}|\\
				& \hspace{0.5 cm} - \left[ 2k\left(\frac{p-1}{2}-k\right)+|\Omega_1^{-1}|+|\Omega_2^{-1}|+\left(2k-\frac{p-1}{2}\right)^2+\epsilon\right]\\
				& =   k^2+\left(\frac{p-1}{2}\right)^2-2\left[k\left(\frac{p-1}{2}\right)\right]+k^2-2k\left(\frac{p-1}{2}\right)+2k^2\\
				& \hspace{0.5 cm} - (2k)^2+2\left[2k\left(\frac{p-1}{2}\right)\right]-\left(\frac{p-1}{2}\right)^2+\epsilon\\
				& =4k^2-k\left(\frac{p-1}{2}\right)-k\left(\frac{p-1}{2}\right)-4k^2+2k\left(\frac{p-1}{2}\right)+\epsilon\\
				&=\epsilon.
			\end{align*}
			Since $\epsilon\in\{-1,0,1\}$, it follows that $|e_{f_p^*}(0)-e_{f_p^*}(1)|\leq1.$ Accordingly, $G_1+G_2$ is an arithmetic cordial graph modulo $p$ under $\left\langle S,\zeta_p,\cdot\right\rangle$.
		\end{proof}
		
		\tab An an example of Theorem~\ref{thmjoin}, consider the following result.
		
		\begin{theorem}\normalfont
			Assume that $\frac{p-1}{2}\equiv 0\pmod{4}$. The join graph $P_{(p-1)/2}+C_{(p-1)/2}$ is an arithmetic cordial graph modulo $p$ under $\left\langle S,\zeta_p,\cdot\right\rangle$, where $S=\{1,2,\ldots,p-1\}$ and $p\geq 17$.
		\end{theorem}

		\begin{proof}
			Since $\frac{p-1}{2}\equiv 0\pmod{4}$, we have $\frac{p-1}{2}=4k$ for some integer $k$. We will use Theorem~\ref{thmjoin} with $G_1=P_{(p-1)/2}$ and $G_2=C_{(p-1)/2}$. Now, assume that $V(P_{4k})=\{v_1,v_2,\ldots,v_{4k}\}$ and $V(C_{4k})=\{u_1,u_2,\ldots,u_{4k}\}$. Define a function $f:V(P_{4k}+C_{4k})\to \{1,2,\ldots,8k\}$ as
			\begin{align*}
				f(v_i)&=\begin{cases}
					r_i&\text{ for }i\equiv 1\text{ or }2\pmod{4}\\
					s_i&\text{ for }i\equiv 0\text{ or }3\pmod{4},
				\end{cases}\\
				f(u_i)&=\begin{cases}
					s_i&\text{ for }i\equiv 1\text{ or }2\pmod{4}\\
					r_i&\text{ for }i\equiv 0\text{ or }3\pmod{4},
				\end{cases}
			\end{align*}
			for $i=1,2,\ldots,4k$. So, $f$ is a bijective function and $|B|=k$. 
			
			\tab For the edges of $P_{4k}$,
			$$f(v_i)f(v_{i+1})=\begin{cases}
				r_ir_{i+1}&\text{ for }i\equiv 1\pmod{4}\\
				r_is_{i+1}&\text{ for }i\equiv 2\pmod{4}\\
				s_is_{i+1}&\text{ for }i\equiv 3\pmod{4}\\
				s_ir_{i+1}&\text{ for }i\equiv 0\pmod{4}
			\end{cases}$$ 
			for $i=1,2,\ldots,4k-1$. By Remark~\ref{rem5.1}, equation (\ref{eqrem}),
			$$\Omega_{1}^{1}=\bigcup_{\substack{i=1\\i\equiv 1\text{ or }3\pmod{4}}}^{4k-1}\{v_iv_{i+1}\} \text{ and }\Omega_{1}^{-1}=\bigcup_{\substack{i=1\\i\equiv 0\text{ or }2\pmod{4}}}^{4k-1}\{v_iv_{i+1}\}.$$
			Hence, $|\Omega_1^1|=2k$ and $|\Omega_1^{-1}|=2k-1$.
			
			\tab For the edges of $C_{4k}$, 
			$$f(u_i)f(u_{i+1})=\begin{cases}
				s_is_{i+1}&\text{ for }i\equiv 1\pmod{4}\\
				s_ir_{i+1}&\text{ for }i\equiv 2\pmod{4}\\
				r_ir_{i+1}&\text{ for }i\equiv 3\pmod{4}\\
				r_is_{i+1}&\text{ for }i\equiv 0\pmod{4}
			\end{cases}$$ 
			for $i=1,2,\ldots,4k-1$, and
			$$f(u_1)f(u_{4k})\equiv s_1r_{4k}.$$
			Using Remark~\ref{rem5.1}, equation~(\ref{eqrem}), we have
			$$\Omega_{2}^{1}=\bigcup_{\substack{i=1\\i\equiv 1\text{ or }3\pmod{4}}}^{4k-1}\{v_iv_{i+1}\} \text{ and }\Omega_{2}^{-1}=\left[\bigcup_{\substack{i=1\\i\equiv 0\text{ or }2\pmod{4}}}^{4k-1}\{v_iv_{i+1}\}\right]\cup \{u_1u_{4k}\}.$$
			Thus, $|\Omega_2^1|=|\Omega_2^{-1}|=2k$.
			
			\tab Applying equation~\ref{eqjoin} with $\epsilon=1$, we have
			\begin{align*}
				2k+2k&=2k-1+2k+\left(2k-2k\right)^2+1\\
				4k&=4k.
			\end{align*}
			Therefore, $P_{(p-1)/2}+C_{(p-1)/2}$ is an arithmetic cordial graph modulo $p$ under $\left\langle S,\zeta_p,\cdot\right\rangle$, where $S=\{1,2,\ldots,p-1\}$.
		\end{proof}
		
		\begin{theorem}\normalfont
			Let $G$ be a connected graph of order $\frac{p-1}{2}$ and size $\frac{p-1}{2}+\epsilon$, where $\epsilon\in \{-1,0,1\}$. Then the corona graph $G\circ K_1$ is an arithmetic cordial graph modulo $p$ under $\left\langle S,\zeta_p,\cdot\right\rangle$, where $S=\{1,2,\ldots,p-1\}$, for all $p\geq 3$.
		\end{theorem}
		
		\begin{proof}
			Let $V(G)=\{v_1,v_2,\ldots,v_{(p-1)/2}\}$. Suppose that $K_1^i$ is the $i$th copy of $K_1$ with $V(K_1^i)=\{u_i\}$ such that $u_i$ is adjacent to $v_i$, for $i=1,2,\ldots,\frac{p-1}{2}$. Assume that $f:V(G\circ K_1)\to S$ is a function defined by
			$$f(v_i)=r_i\text{ and }f(u_i)=s_i$$
			for $i=1,2,\ldots,\frac{p-1}{2}$, where $S=\{1,2,\ldots,p-1\}$. 
			
			\tab For the edges of $G$, if $v_iv_j\in E(G)$, for some $i,j\in \{1,2,\ldots,\frac{p-1}{2}\}$, then 
			$$f(v_i)f(v_j)=r_ir_j.$$
			By Remark~\ref{rem5.1}, $f_p^*(e)=1$ for every $e\in E(G)$.
			
			\tab For the rest of the edges,
			$$f(u_i)f(v_i)=s_ir_i$$
			for $i=1,2,\ldots,\frac{p-1}{2}$. Thus, by Remark~\ref{rem5.1}, $f_p^*(u_iv_i)=0$ for $i=1,2,\ldots,\frac{p-1}{2}$. 
			
			\tab Therefore,
			$$e_{f_p^*}(0)=\frac{p-1}{2}\text{ and }e_{f_p^*}(1)=\frac{p-1}{2}+\epsilon.$$
			Since $\epsilon\in \{-1,0,1\}$, it follows that $|e_{f_p^*}(0)-e_{f_p^*}(1)|\leq 1$. Hence, $G\circ K_1$ is an arithmetic cordial graph modulo $p$ under $\left\langle S,\zeta_p,\cdot\right\rangle$.
		\end{proof}
		
		\begin{theorem}\normalfont
			Let $G$ be a connected graph of order $p-1$ and size $m(p-1)+\epsilon$, where $\epsilon\in \{-1,0,1\}$. In addition, suppose that $H$ is a graph of order $p-1$ and size $p-1+m$. Then the corona graph $G\circ H$ is an arithmetic cordial graph modulo $p$ under $\left\langle S,\zeta_p,\cdot\right\rangle$, where $S=\{1,2,\ldots,p(p-1)\}$, for all $m\geq 1$ and $p\geq 3$.
		\end{theorem}
		
		\begin{proof}
			Let $V(G)=\{v_1,v_2,\ldots,v_{p-1}\}$. Suppose that $H^i$ is the $i$th copy of $H$ with $V(H^i)=\{u_1^i,u_2^i,\ldots,u_{p-1}^i\}$ where all vertices of $H^i$ is adjacent to $v_i$ for $i=1,2,\ldots,p-1$. Define a function $f:V(G\circ H)\to S$ as
			\begin{align*}
				f(u_j^i)&=\begin{cases}
					r_i+(j-1)p&\text{ for }i=1,2,\ldots,\frac{p-1}{2}\\
					s_{i-(p-1)/2}+(j-1)p&\text{ for }i=\frac{p+1}{2},\frac{p+3}{2},\ldots,p-1,
				\end{cases}
			\end{align*}
			for $j=1,2,\ldots,p-1$, and 
			$$f(v_i)=p+(i-1)p\text{ for }i=1,2,\ldots,p-1,$$
			where $S=\{1,2,\ldots,p-1\}$. Thus, $f$ is a bijective function.
			
			\tab For the edges of $G$, if $v_iv_j\in E(G)$, for some $i,j\in\{1,2,\ldots,p-1\}$, then
			$$f(v_i)f(v_j)\equiv 0\pmod{p}.$$
			Hence, $f_p^*(e)=0$ for all $e\in E(G)$.
			
			\tab For the edges of $H^i$, if $u_a^iu_b^i\in E(H^i)$, for some $a,b\in \{1,2,\ldots,p-1\}$, we have
			$$f(u_a^i)f(u_b^i)\equiv \begin{cases}
				r_i^2\pmod{p}&\text{ for }i=1,2,\ldots,\frac{p-1}{2}\\
				s_{i-(p-1)/2}^2\pmod{p}&\text{ for }i=\frac{p+1}{2},\frac{p+3}{2},\ldots,p-1.
			\end{cases}$$
			By Remark~\ref{rem5.1}, $f_p^*(e)=1$ for all $e\in E(H^i)$, for $i=1,2,\ldots,p-1$.
			
			\tab For the remaining edges, for each $i,j=1,2,\ldots,p-1$,
			$$f(v_i)f(u_j^i)\equiv 0\pmod{p}$$
			and it follows that
			$$f_p^*(v_iu_j^i)=0.$$
			
			\tab Consequently,
			$$e_{f_p^*}(0)=m(p-1)+\epsilon+(p-1)^2\text{ and }e_{f_p^*}(1)=(p-1)(p-1+m)=(p-1)^2+m(p-1).$$
			Because $\epsilon\in\{-1,0,1\}$, we have $|e_{f_p^*}(0)-e_{f_p^*}(1)|\leq 1$. Therefore, $G\circ H$ is an arithmetic cordial graph modulo $p$ under $\left\langle S,\zeta_p,\cdot\right\rangle$.
		\end{proof}
		
		\section*{Acknowledgement}
		\tab The authors gratefully acknowledge the support of the Department of Science and Technology - Science Education Institute (DOST-SEI) under the STRAND Scholarship Program.
		
		\section*{Conflict of Interest}
		\tab The authors declare that there is no conflict of interest.

\end{document}